\documentclass[11pt,twoside,a4paper]{article}
\usepackage{color,amscd,amsmath,amssymb,amsthm,latexsym,stmaryrd,authblk,mathabx,shuffle,mathrsfs,hyperref,tikz,
tkz-tab} 
\usepackage[latin1]{inputenc}
\usepackage[T1]{fontenc}   
\usepackage[english]{babel}
\providecommand{\AMSclass}[1]{\textbf{\textit{AMS classification.}} #1}

\input{xy}
\xyoption{all}

\title{The Fortuin-Kasteleyn polynomial as a bialgebra morphism and applications to the Tutte polynomial}
\date{}
\author{Lo\"ic Foissy}
\affil{\small{Univ. Littoral C\^ote d'Opale, UR 2597
LMPA, Laboratoire de Math\'ematiques Pures et Appliqu\'ees Joseph Liouville
F-62100 Calais, France}.\\ Email: \texttt{foissy@univ-littoral.fr}}
\author{Claudia Malvenuto}
\affil{\small{Dipartimento di Matematica ``Guido Castelnuovo'', Sapienza Universit\`a di Roma\\  P.le A. Moro, 5 00185 Rome, Italy.}\\Email: \texttt{malvenuto@mat.uniroma1.it}}

\newcommand{\tun}{\begin{picture}(5,0)(-2,-1)
\put(0,0){\circle*{2}}
\end{picture}}

\newcommand{\grun}{\begin{tikzpicture}[line cap=round,line join=round,>=triangle 45,x=0.5cm,y=0.5cm]
\clip(-0.2,-0.1) rectangle (0.2,0.2);
\begin{scriptsize}
\draw [fill=black] (0.,0.) circle (1pt);
\end{scriptsize}
\end{tikzpicture}}

\newcommand{\grdeux}{\begin{tikzpicture}[line cap=round,line join=round,>=triangle 45,x=0.5cm,y=0.5cm]
\clip(-.2,-.1) rectangle (0.2,0.7);
\draw [line width=.5pt] (0.,0.5)-- (0.,0.);
\begin{scriptsize}
\draw [fill=black] (0.,0.) circle (1pt);
\draw [fill=black] (0.,0.5) circle (1pt);
\end{scriptsize}
\end{tikzpicture}}

\newcommand{\grtroisun}{\begin{tikzpicture}[line cap=round,line join=round,>=triangle 45,x=0.5cm,y=0.5cm]
\clip(-0.5,-0.1) rectangle (0.5,0.7);
\draw [line width=0.5pt] (0.,0.)-- (-0.3,0.5);
\draw [line width=0.5pt] (0.,0.)-- (0.3,0.5);
\draw [line width=0.5pt] (-0.3,0.5)-- (0.3,0.5);
\begin{scriptsize}
\draw [fill=black] (-0.3,0.5) circle (1pt);
\draw [fill=black] (0.,0.) circle (1pt);
\draw [fill=black] (0.3,0.5) circle (1pt);
\end{scriptsize}
\end{tikzpicture}}
\newcommand{\grtroisdeux}{\begin{tikzpicture}[line cap=round,line join=round,>=triangle 45,x=0.5cm,y=0.5cm]
\clip(-0.5,-0.1) rectangle (0.5,0.7);
\draw [line width=0.5pt] (0.,0.)-- (-0.3,0.5);
\draw [line width=0.5pt] (0.,0.)-- (0.3,0.5);
\begin{scriptsize}
\draw [fill=black] (-0.3,0.5) circle (1pt);
\draw [fill=black] (0.,0.) circle (1pt);
\draw [fill=black] (0.3,0.5) circle (1pt);
\end{scriptsize}
\end{tikzpicture}}

\newcommand{\grquatreun}{\begin{tikzpicture}[line cap=round,line join=round,>=triangle 45,x=0.5cm,y=0.5cm]
\clip(-0.2,-0.1) rectangle (0.7,0.7);
\begin{scriptsize}
\draw [fill=black] (0.,0.) circle (1pt);
\draw [fill=black] (0.,0.5) circle (1pt);
\draw [fill=black] (0.5,0.5) circle (1pt);
\draw [line width=.5pt] (0.,0.)-- (0.5,0.);
\draw [line width=0.5pt] (0.,0.)-- (0.,0.5);
\draw [line width=0.5pt] (0.5,0.)-- (0.5,0.5);
\draw [line width=0.5pt] (0.0,0.)-- (0.5,0.5);
\draw [line width=0.5pt] (0.0,0.5)-- (0.5,0.5);
\draw [line width=0.5pt] (0.0,0.5)-- (0.5,0.);
\draw [fill=black] (0.5,0.) circle (1pt);
\end{scriptsize}
\end{tikzpicture}}
\newcommand{\grquatredeux}{\begin{tikzpicture}[line cap=round,line join=round,>=triangle 45,x=0.5cm,y=0.5cm]
\clip(-0.2,-0.1) rectangle (0.7,0.7);
\begin{scriptsize}
\draw [fill=black] (0.,0.) circle (1pt);
\draw [fill=black] (0.,0.5) circle (1pt);
\draw [fill=black] (0.5,0.5) circle (1pt);
\draw [line width=.5pt] (0.,0.)-- (0.5,0.);
\draw [line width=0.5pt] (0.,0.)-- (0.,0.5);
\draw [line width=0.5pt] (0.5,0.)-- (0.5,0.5);
\draw [line width=0.5pt] (0.,0.)-- (0.5,0.5);
\draw [line width=0.5pt] (0.,0.5)-- (0.5,0.5);
\draw [fill=black] (0.5,0.) circle (1pt);
\end{scriptsize}
\end{tikzpicture}}
\newcommand{\grquatretrois}{\begin{tikzpicture}[line cap=round,line join=round,>=triangle 45,x=0.5cm,y=0.5cm]
\clip(-0.2,-0.1) rectangle (0.7,0.7);
\begin{scriptsize}
\draw [fill=black] (0.,0.) circle (1pt);
\draw [fill=black] (0.,0.5) circle (1pt);
\draw [fill=black] (0.5,0.5) circle (1pt);
\draw [line width=.5pt] (0.,0.)-- (0.5,0.);
\draw [line width=0.5pt] (0.,0.)-- (0.,0.5);
\draw [line width=0.5pt] (0.5,0.)-- (0.5,0.5);
\draw [line width=0.5pt] (0.,0.)-- (0.5,0.5);
\draw [fill=black] (0.5,0.) circle (1pt);
\end{scriptsize}
\end{tikzpicture}}
\newcommand{\grquatrequatre}{\begin{tikzpicture}[line cap=round,line join=round,>=triangle 45,x=0.5cm,y=0.5cm]
\clip(-0.2,-0.1) rectangle (0.7,0.7);
\begin{scriptsize}
\draw [fill=black] (0.,0.) circle (1pt);
\draw [fill=black] (0.,0.5) circle (1pt);
\draw [fill=black] (0.5,0.5) circle (1pt);
\draw [line width=.5pt] (0.,0.)-- (0.5,0.);
\draw [line width=0.5pt] (0.,0.)-- (0.,0.5);
\draw [line width=0.5pt] (0.5,0.)-- (0.5,0.5);
\draw [line width=0.5pt] (0.,0.5)-- (0.5,0.5);
\draw [fill=black] (0.5,0.) circle (1pt);
\end{scriptsize}
\end{tikzpicture}}
\newcommand{\grquatrecinq}{\begin{tikzpicture}[line cap=round,line join=round,>=triangle 45,x=0.5cm,y=0.5cm]
\clip(-0.2,-0.1) rectangle (0.7,0.7);
\begin{scriptsize}
\draw [fill=black] (0.,0.) circle (1pt);
\draw [fill=black] (0.,0.5) circle (1pt);
\draw [fill=black] (0.5,0.5) circle (1pt);
\draw [line width=.5pt] (0.,0.)-- (0.5,0.);
\draw [line width=0.5pt] (0.,0.)-- (0.,0.5);
\draw [line width=0.5pt] (0.,0.)-- (0.5,0.5);
\draw [fill=black] (0.5,0.) circle (1pt);
\end{scriptsize}
\end{tikzpicture}}
\newcommand{\grquatresix}{\begin{tikzpicture}[line cap=round,line join=round,>=triangle 45,x=0.5cm,y=0.5cm]
\clip(-0.2,-0.1) rectangle (0.7,0.7);
\begin{scriptsize}
\draw [fill=black] (0.,0.) circle (1pt);
\draw [fill=black] (0.,0.5) circle (1pt);
\draw [fill=black] (0.5,0.5) circle (1pt);
\draw [line width=.5pt] (0.,0.)-- (0.5,0.);
\draw [line width=0.5pt] (0.,0.)-- (0.,0.5);
\draw [line width=0.5pt] (0.5,0.)-- (0.5,0.5);
\draw [fill=black] (0.5,0.) circle (1pt);
\end{scriptsize}
\end{tikzpicture}}

\newcommand{\grdeuxo}{\begin{tikzpicture}[line cap=round,line join=round,>=triangle 45,x=.4cm,y=0.4cm]
\clip(-0.3,-0.1) rectangle (0.3,1.2);
\draw[->]  (0.,0.)-- (0.,1.);
\draw [fill=black] (0.,0.) circle (.8pt);
\draw [fill=black] (0.,1.) circle (.8pt);
\end{tikzpicture}}

\newcommand{\cycle}{\begin{tikzpicture}[line cap=round,line join=round,>=triangle 45,x=0.4cm,y=0.4cm]
\clip(3.2,-0.2) rectangle (4.6,2.2);
\draw [shift={(5.,1.)},line width=.5pt]  plot[domain=2.356194490192345:3.9269908169872414,variable=\t]({1.*1.4142135623730951*cos(\t r)+0.*1.4142135623730951*sin(\t r)},{0.*1.4142135623730951*cos(\t r)+1.*1.4142135623730951*sin(\t r)});
\draw [shift={(3.,1.)},line width=.5pt]  plot[domain=-0.7853981633974483:0.7853981633974483,variable=\t]({1.*1.4142135623730951*cos(\t r)+0.*1.4142135623730951*sin(\t r)},{0.*1.4142135623730951*cos(\t r)+1.*1.4142135623730951*sin(\t r)});
\draw [->,line width=.5pt] (4.1,1.9) -- (4.,2.);
\draw [->,line width=.5pt] (3.9,0.1) -- (4.,0.);
\begin{scriptsize}
\draw [fill=black] (4.,0.) circle (.8pt);
\draw [fill=black] (4.,2.) circle (.8pt);
\end{scriptsize}
\end{tikzpicture}}

\newcommand{\grtrois}{\begin{tikzpicture}[line cap=round,line join=round,>=triangle 45,x=.4cm,y=0.4cm]
\clip(-0.1,-0.1) rectangle (1.1,1.2);
\draw  (0.5,0.)-- (0.,1.);
\draw  (0.,1.)-- (1.,1.);
\draw (1.,1.)-- (0.5,0.);
\draw [fill=black] (0.5,0.) circle (.8pt);
\draw [fill=black] (0.,1.) circle (.8pt);
\draw [fill=black] (1.,1.) circle (.8pt);
\end{tikzpicture}}

\newcommand{\grtroisoun}{\begin{tikzpicture}[line cap=round,line join=round,>=triangle 45,x=.4cm,y=0.4cm]
\clip(-0.1,-0.1) rectangle (1.1,1.2);
\draw[->]  (0.5,0.)-- (0.,1.);
\draw [->] (0.,1.)-- (1.,1.);
\draw[->] (1.,1.)-- (0.5,0.);
\draw [fill=black] (0.5,0.) circle (.8pt);
\draw [fill=black] (0.,1.) circle (.8pt);
\draw [fill=black] (1.,1.) circle (.8pt);
\end{tikzpicture}}

\newcommand{\grtroisodeux}{\begin{tikzpicture}[line cap=round,line join=round,>=triangle 45,x=.4cm,y=0.4cm]
\clip(-0.1,-0.1) rectangle (1.1,1.2);
\draw[->]  (0.5,0.)-- (0.,1.);
\draw [->] (0.,1.)-- (1.,1.);
\draw[<-] (1.,1.)-- (0.5,0.);
\draw [fill=black] (0.5,0.) circle (.8pt);
\draw [fill=black] (0.,1.) circle (.8pt);
\draw [fill=black] (1.,1.) circle (.8pt);
\end{tikzpicture}}

\theoremstyle{plain}
\newtheorem{theo}{Theorem}[section]
\newtheorem{lemma}[theo]{Lemma}
\newtheorem{cor}[theo]{Corollary}
\newtheorem{prop}[theo]{Proposition}
\newtheorem{defi}[theo]{Definition}

\theoremstyle{remark}
\newtheorem{remark}{Remark}[section]
\newtheorem{notation}{Notations}[section]
\newtheorem{example}{Example}[section]

\newcommand{\K}{\mathbb{K}}
\newcommand{\N}{\mathbb{N}}
\newcommand{\Z}{\mathbb{Z}}

\newcommand{\bfG}{\mathbf{G}}

\newcommand{\id}{\mathrm{Id}}

\newcommand{\eq}{\mathcal{E}}
\newcommand{\cl}{\mathrm{cl}}

\newcommand{\Char}{\mathrm{Char}}

\newcommand{\tdelta}{\tilde{\Delta}}

\renewcommand{\ker}{\mathrm{Ker}}

\newcommand{\cc}{\mathrm{cc}}

\newcommand{\calH}{\mathcal{H}}
\newcommand{\bfH}{\mathbf{H}}

\newcommand{\calA}{\mathcal{A}}
\newcommand{\calB}{\mathcal{B}}

\newcommand{\chara}{\mathrm{Char}}

\renewcommand{\ker}{\mathrm{Ker}}

\newcommand{\gr}{\mathrm{gr}}
\newcommand{\potac}{\mathcal{PO}_{\mathit{tac}}}
\newcommand{\cov}{\mathcal{C}}
\newcommand{\covfor}{\mathcal{CF}}
\newcommand{\spann}{\mathcal{S}}
\newcommand{\spanfor}{\mathcal{SF}}
\newcommand{\ori}{\mathcal{O}}
\newcommand{\oriac}{\mathcal{O}_{\mathit{ac}}}

\newcommand{\opc}{\mathcal{OPC}}
\newcommand{\osc}{\mathcal{O}_{sc}}

\newcommand{\pc}{\mathcal{PC}}
 
\begin{document}

\maketitle

\begin{abstract}
We compute an explicit formula for the antipode of the double bialgebra of graphs \cite{Foissy36} in terms of totally acyclic partial orientations, using some general results on double bialgebras \cite{Foissy40}. In analogy to what was already  proven in Hopf-algebraic terms for the chromatic polynomial of a graph, we show that the Fortuin-Kasteleyn polynomial (a variant of the Tutte polynomial) is a morphism of the double algebra of graphs into that of polynomials, which generalizes the chromatic polynomial. When specialized at particular values, we give  combinatorial interpretations of the Tutte  polynomial of a graph, via covering graphs and covering forests, and of the  Fortuin-Kasteleyn polynomial, via pairs of vertex--edge colorings. Finally we show that the map associating to a graph all its orientations is a Hopf morphism from the double bialgebra of graphs into that one of oriented graphs, allowing to give interpretations of the  Fortuin-Kasteleyn polynomial when computed at negative values.
\end{abstract}

\AMSclass{16T05 05C25}

\tableofcontents

\section*{Introduction}

In \cite{Birkhoff1912} George D. Birkhoff  introduced the notion of chromatic polynomial. It counts the number of proper graph colorings as a function of the number of colors. Later  \cite{Birkhoff1946} Birkhoff and Lewis studied it extensively in the restricted case of planar graphs in the attempt to solve the so--called ``four color problem''.  In 1932, Hassler Whitney  \cite{Whitney1932} and then William Tutte  \cite{Tutte1954}  generalized the chromatic polynomial to a new polynomial, which Tutte called the dichromate of a graph but it is better known as the Tutte (or Tutte--Whitney) polynomial, and plays an important role in graph theory. It is a polynomial in two variables associated to an undirected graph and tells a lot of information about the way a graph is connected. It actually  contains several specializations from other domains, such as the Jones polynomial from knot theory and the partition functions of the Potts model from statistical physics.  Thanks to his results, it is observed there that the theory of spanning trees links the theory of graph--colorings to that of electrical networks. Another polynomial related to it is the Fortuin--Kasteleyn polynomial: it appears in  the random-cluster model, a generalization both of the percolation model and of the Ising model. Cees Fortuin and Cornelius Kasteleyn in \cite{Fortuin1971, Fortuin1972}  show that the theory of the random-cluster model is intimately connected with the combinatorial theory of graphs. The physical question for a medium with randomly distributed pores through which a liquid percolates is modeled mathematically as a three-dimensional network of $n \times n \times  n$ vertices, usually called ``sites'', in which the edge or ``bonds'' between each two neighbors may be open (allowing the liquid through) with probability $p$, or closed with probability $1 - p$, and they are assumed to be independent. Therefore, for a given $p$, it is natural to ask for the probability that an open path (each of whose links is an ``open'' bond) exists from the top to the bottom, and for the behavior of the system for large $n$. This problem, called  bond percolation model, was introduced in 1957  by Broadbent and Hammersley \cite{Broadbent1957} and has been studied intensively by mathematicians and physicists since then.

Since the seminal work of \cite{Joni1982} it was clear that standard operations/transformations on combinatorial objects have an algebraic nature, and even more coalgebras operations of cutting the objects into smaller pieces have strong significance. In fact the three graph polynomials above mentioned satisfy a recurrence based on some graph operations called deletion or contraction of edges. The advantage of the Tutte and the Fortuin-Kasteleyn polynomials is that during the recursive process they loose much less informations than the chromatic polynomial. The way we treat these polynomials and their significance in graph theory, however, takes a different approach: that of combinatorial Hopf algebras. 

Several notable results that go in this direction are contained in the first author's work \cite{Foissy36}, where the chromatic polynomial is characterized as the unique polynomial invariant of graphs, compatible with two interacting bialgebras structures on graphs: the first coproduct is given by partitions of vertices into two parts, the second one by a contraction--extraction process. This gives Hopf-algebraic proofs of Rota's result on the signs of coefficients of chromatic polynomials and of Stanley's interpretation of the values at negative integers of chromatic polynomials. 

The aim of the present work is to  study this phenomenon in more detail, to extend it to the other polynomials mentioned inserting them  into the theory of combinatorial Hopf algebras, and to recover new simpler proofs of  classical results on Tutte and Fortuin-Kasteleyn polynomials. 

We start recalling in  Section \ref{sec:RemDoubleAlg} the notion of double bialgebras due to the first author (which appears in \cite{Foissy40} as cointeracting bialgebras), that is to say bialgebras with two coproducts, the first one being a comodule morphism for the coaction induced by the second one, and the main related results on the monoid of characters and their actions that will be used later. In Section \ref{sec:RemHopfGraphs} we describe the central example of a double bialgebra on graphs, as introduced in \cite{Foissy36} and a variation on oriented graphs, in which the first coproduct comes from splitting a graph $G$ in two parts, taking the induced subgraphs via a bipartition of its vertices. The second coproduct comes from a ``deletion/contraction'' process. We are able to express the antipode of the second Hopf algebra. We then go to  the Fortuin--Kasteleyn polynomial, seen as a Hopf morphism of the Hopf algebra of graphs into the Hopf algebra of polynomials in one indeterminate, and express some invariants of graphs as characters making use of the general results on bialgebras. The so--developed algebraic frame allows us to give in Section \ref{sec:CombInterp} some combinatorial interpretation to both Tutte and Fortuin-Kasteleyn polynomials when specialized at particular values: these results are analogue of the statement in \cite{Foissy36} that the chromatic polynomial of a graph is the unique polynomial invariant on graphs compatible with both bialgebraic structures. We end with some result on orientations of graphs, in Section  \ref{sec:Orientations}: the map associating to a graph the sum of all possible orientations on G is a Hopf--morphism, and we recover -- giving easy proofs  -- some known but complex results on specializations of Tutte's polynomial.\\

\textbf{Acknowledgments}. The author acknowledges support from the grant ANR-20-CE40-0007 \emph{Combinatoire Algébrique, Résurgence, Probabilités Libres et Opérades}.\\

\begin{notation} \begin{enumerate}
\item We denote by $\K$ a commutative field of characteristic $0$. Any vector space in this field will be taken over $\K$.
\item For any $n\in \N$, we denote by $[n]$ the set the first $n$ strictly positive integers $\{1,\ldots,n\}$. In particular, $[0]=\emptyset$.
\end{enumerate}\end{notation}

\section{Reminders on double bialgebras} \label{sec:RemDoubleAlg}

\begin{defi}
A double bialgebra is a quadruple $(A,m,\Delta,\delta)$ with a product  $m: A\otimes A\rightarrow A $ and two coproducts 
$\Delta, \delta : A\rightarrow A\otimes A$ such that:
\begin{itemize}
\item $(A,m,\Delta)$ is a bialgebra. Its counit is denoted by $\varepsilon_\Delta$.
\item $(A,m,\delta)$ is a bialgebra. Its counit  is denoted by $\epsilon_\delta$.
\item The following compatibilities are satisfied:
\begin{align*}
(\varepsilon_\Delta\otimes \id)\circ \delta&=\eta\circ \varepsilon_\Delta,\\
(\Delta \otimes \id)\circ \delta&=m_{1,3,24}\circ (\delta \otimes \delta)\circ \Delta,
\end{align*}
where
\begin{align*}
\eta&:\left\{\begin{array}{rcl}
\K&\longrightarrow&A\\
x&\longmapsto&x1_A,
\end{array}\right.&
m_{1,3,24}&:\left\{\begin{array}{rcl}
A^{\otimes 4}&\longrightarrow&A^{\otimes 3}\\
a\otimes b\otimes c\otimes d&\longmapsto&a\otimes c\otimes bd.
\end{array}\right.
\end{align*}
\end{itemize}
\end{defi}

\begin{example}
The polynomial algebra $\K[X]$ is a double bialgebra, with the two multiplicative coproducts $\Delta$ and $\delta$, defined by
\begin{align*}
\Delta(X)&=X\otimes 1+1\otimes X,&\delta(X)&=X\otimes X. 
\end{align*}
The two counits are given by
\begin{align*}
&\forall P\in \K[X],&\varepsilon_\Delta(P)&=P(0),&\epsilon_\delta(P)&=P(1). 
\end{align*}
\end{example}

If $(A,m,\Delta,\delta)$ is a double bialgebra, 
its set of characters $\chara(A)$ inherits two convolution products, each making $\chara(A)$ a monoid:
\begin{align*}
&\forall \lambda,\mu\in \chara(A),&\lambda *\mu&=(\lambda \otimes \mu)\circ \Delta,
&\lambda \star\mu&=(\lambda \otimes \mu)\circ \delta.
\end{align*}
The compatibility between $\Delta$ and $\delta$ implies that
\begin{align*}
&\forall \lambda,\mu,\nu\in \chara(A),&(\lambda*\mu)\star \nu&=(\lambda \star \nu)*(\mu\star \nu).
\end{align*}

\begin{theo} \cite[Proposition 2.5]{Foissy40}
Let $(A,m,\Delta,\delta)$ be a double bialgebra and $(B,m,\Delta)$ be a bialgebra. We denote by $E_{A\rightarrow B}$
the set of bialgebra morphisms from $(A,m,\Delta)$ to $(B,m,\Delta)$.  The following map is an action of the monoid
$(\chara(A),\star)$ onto $E_{A\rightarrow B}$:
\begin{align*}
&\left\{\begin{array}{rcl}
E_{A\rightarrow B}\otimes \chara(A)&\longrightarrow&E_{A\rightarrow B}\\
(\phi,\lambda)&\longmapsto&\phi\leftsquigarrow \lambda=(\phi\otimes \lambda)\circ \delta.
\end{array}\right.
\end{align*}
\end{theo}

\begin{notation}
Let $(A,m,\Delta)$ be a bialgebra. We denote by $A_+=\ker(\varepsilon_\Delta)$ its augmentation ideal.
We denote by $\tdelta:A_+\longrightarrow A_+\otimes A_+$ defined by
\begin{align*}
&\forall a\in A_+,&\tdelta(a)&=\Delta(a)-a\otimes 1_A-1_A\otimes a.
\end{align*}
It is coassociative, and we can consider its iterations $\tdelta^{(n)}:A_+\longrightarrow A_+^{\otimes (n+1)}$,
inductively defined by 
\begin{align*}
\tdelta^{(n)}&=\begin{cases}
\id_{A_+}\mbox{ if }n=0,\\
(\tdelta^{(n-1)}\otimes \id)\circ \tdelta\mbox{ if }n\geq 1.
\end{cases}
\end{align*}
We shall say that $(A,m,\Delta)$ is connected if 
\[A_+=\bigcup_{n\geq 0} \ker(\tdelta^{(n)}).\]
If $(A,m,\Delta,\delta)$ is a double bialgebra, we shall say that it is connected if $(A,m,\Delta)$ is connected.
\end{notation}

\begin{theo}\label{theoaajouter} \cite[Theorem 3.9 and Corollary 3.11]{Foissy40}
Let $(A,m,\Delta,\delta)$ be a connected double bialgebra. 
\begin{enumerate}
\item For any $\lambda\in \chara(A)$, there exists a unique $\phi_\lambda\in E_{A\rightarrow \K[X]}$ such that 
$\epsilon_\delta \circ \phi_\lambda=\lambda$. For any $a\in A_+$,
\[\phi_\lambda(a)=\sum_{n=1}^\infty \lambda^{\otimes n}\circ \tdelta^{(n-1)}(a) \frac{X(X-1)\ldots (X-n+1)}{n!}.\]
\item The morphism $\phi_{\epsilon_\delta}$ is the unique double bialgebra morphism from $(A,m,\Delta,\delta)$
to $(\K[X],\Delta,\delta)$.
\item The two following maps are bijections, inverse one from the other:
\begin{align*}
&\left\{\begin{array}{rcl}
E_{A\rightarrow \K[X]}&\longrightarrow&\chara(A)\\
\phi&\longmapsto&\epsilon_\delta \circ \phi.
\end{array}\right.&
&\left\{\begin{array}{rcl}
\chara(A)&\longrightarrow&E_{A\rightarrow \K[X]}\\
\lambda&\longrightarrow&\phi_\lambda=\phi_{\epsilon_\delta}\leftsquigarrow \lambda.
\end{array}\right.&
\end{align*}
\end{enumerate}
\end{theo}

\begin{theo}\cite[Corollary 2.3]{Foissy40} \label{theoantipode}
Let $(A,m,\Delta,\delta)$ be a double bialgebra, such that $(A,m,\Delta)$ is a Hopf algebra. 
We denote by $\epsilon_\delta^{*-1}$ the inverse of the character $\epsilon_\delta$ for the convolution associated to $\Delta$.
Then the antipode of $(A,m,\Delta)$ is given by
\[S=(\epsilon_\delta^{*-1}\otimes \id)\circ \delta.\]
If $\phi:(A,m,\Delta,\delta)\longrightarrow (\K[X],m,\Delta,\delta)$ is a double bialgebra morphism, then for any $x\in A$,
\[\epsilon_\delta^{*-1}(x)=\phi(x)(-1).\]
\end{theo}

\section{Bialgebraic structure on graphs} \label{sec:RemHopfGraphs}

\subsection{Products and coproducts}

We refer to  \cite{Bollobas1998}  and \cite{Harary1969} for classical definitions and notations on graphs.
A (simple) graph is a pair  $G=(V(G), E(G))$, where $V(G)$ is a set, whose elements are called the vertices of the graph, and  $E(G)\subseteq \binom{V(G)}{2}$ is a subset of unordered pairs of vertices, called the edges of $G$. 
We shall denote by $\bfG$ the set of (isoclasses of) graphs. The vector space generated by $\bfG$ will be denoted by $\calH_\bfG$. 

\begin{example} 
\begin{align*}
\bfG&=\left\{\begin{array}{c}
1,\grun,\grdeux,\grun\grun,\grtroisun,\grtroisdeux,\grdeux\grun,\grun\grun\grun,\\
\grquatreun,\grquatredeux,\grquatretrois,\grquatrequatre,\grquatrecinq,\grquatresix,
\grtroisun\grun,\grtroisdeux\grun,\grdeux\grdeux,\grdeux\grun\grun,\grun\grun\grun\grun,\ldots
\end{array}\right\}.
\end{align*}
\end{example}

If $G$ and $H$ are two graphs, their disjoint union is the graph $GH$ defined by 
\begin{align*}
V(GH)&=V(G)\sqcup V(H), &E(GH)&=E(G)\sqcup E(H).
\end{align*}
This yields a commutative and associative product $m$ on $\calH_\bfG$, whose unit is the empty graph $1$.\\

\noindent Let $G$ be a graph and $I\subseteq V(G)$. The subgraph induced in $G$ by $I$, denoted here by  $G_{\mid I}$, is defined by
\begin{align*}
V(G_{\mid I})&=I,&E(G_{\mid I})&=\{\{x,y\}\in E(G)\mid x,y\in I\}.
\end{align*}
Using the notion of induces subgraph,  any bipartition $(I, V(G)\setminus I)$ of vertices ``splits'' the graph into two pieces, $G_{\mid I}$ and  
$G_{\mid V(G)\setminus I}$: this in turn allows to define a cocommutative and coassociative coproduct $\Delta$ on $\calH_\bfG$ (see Schmitt \cite{Schmitt1994}) given by
\begin{align*}
&\forall G\in\bfG,&\Delta(G)&=\sum_{I\subseteq V(G)} G_{\mid I}\otimes G_{\mid V(G)\setminus I}.
\end{align*}
Its counit $\varepsilon_\Delta$ is given by
\begin{align*}
&\forall G\in\bfG,&\varepsilon_\Delta(G)&=\delta_{G,1}.
\end{align*}

\begin{example}
\begin{align*}
\Delta(\grun)&=\grun \otimes 1+1\otimes \grun,\\
\Delta(\grdeux)&=\grdeux\otimes 1+1\otimes \grdeux+2\grun \otimes \grun,\\
\Delta(\grtroisun)&=\grtroisun\otimes 1+1\otimes \grtroisun+3\grun \otimes \grdeux+3\grdeux\otimes \grun,\\
\Delta(\grtroisdeux)&=\grtroisdeux\otimes 1+1\otimes\grtroisdeux+2\grun \otimes \grdeux+
\grun\otimes \grun\grun+2\grdeux\otimes \grun+\grun\grun\otimes \grun,\\
\Delta(\grquatreun)&=\grquatreun\otimes 1+1\otimes\grquatreun
+4\grun \otimes \grtroisun+6\grdeux\otimes \grdeux+4\grtroisun\otimes \grun,\\
\Delta(\grquatredeux)&=\grquatredeux\otimes 1+1\otimes\grquatredeux
+2\grun \otimes \grtroisun+2\grun\otimes \grtroisdeux\\
&+4\grdeux\otimes \grdeux+\grdeux\otimes \grun\grun
+\grun\grun\otimes \grdeux+2\grtroisun\otimes \grun+2\grtroisdeux\otimes \grun,\\
\Delta(\grquatretrois)&=\grquatretrois\otimes 1+1\otimes\grquatretrois
+\grun \otimes \grtroisun+2\grun \otimes \grtroisdeux+\grun \otimes \grdeux\grun\\
&+2\grdeux\otimes \grdeux+2\grun\grun\otimes \grdeux+2\grdeux\otimes \grun\grun
+\grtroisun\otimes \grun+2\grtroisdeux\otimes \grun+\grdeux\grun\otimes \grun,\\
\Delta(\grquatrequatre)&=\grquatrequatre\otimes 1+1\otimes\grquatrequatre
+4\grun \otimes \grtroisdeux+4\grdeux\otimes \grdeux+2\grun\grun\otimes \grun\grun
+4\grtroisdeux\otimes \grun,\\
\Delta(\grquatrecinq)&=\grquatrecinq\otimes 1+1\otimes\grquatrecinq+3\grun \otimes \grtroisdeux
+3\grun\otimes \grun\grun\grun+3\grdeux\otimes \grun\grun+3\grun\grun\otimes \grdeux
+\grtroisdeux\otimes \grun+\grun\otimes \grun\grun\grun,\\
\Delta(\grquatresix)&=\grquatresix\otimes 1+1\otimes\grquatresix
+2\grun\otimes \grtroisdeux+2\grun\otimes\grdeux\grun\\
&+2\grdeux\otimes\grdeux+2\grun\grun\otimes\grun\grun+\grun\grun\otimes\grdeux
+\grdeux\otimes\grun\grun+2\grtroisdeux\otimes \grun+2\grdeux\grun\otimes \grun.
\end{align*}
\end{example}

Let $G$ be a graph and let  $\eq(G)$ the set of equivalence relations on the vertices $V(G)$.  For  $\sim\ \in \eq(G)$, denote by $\pi_\sim:V(G)\longrightarrow V(G)/ \sim$ is the canonical surjection. We define the {\it contracted graph $G/\sim$} by
\begin{align*}
V(G/\sim)&=V(G)/\sim,&E(G/\sim)&=\{\{\pi_\sim(x),\pi_\sim(y)\}\mid \{x,y\}\in E(G), \:\pi_\sim(x)\neq \pi_\sim(y)\}.
\end{align*}

We define the {\it restricted graph} $G\mid\sim$ by
\begin{align*}
V(G\mid \sim)&=V(G),&E(G\mid \sim)=\{\{x,y\}\in E(G)\mid \pi_\sim(x)=\pi_\sim(y)\}.
\end{align*}
In other words, $G\mid \sim$ is the disjoint union of the subgraphs $G_{\mid C}$, with $C\in V(G)/\sim$.
We shall say that an equivalence  $\sim$ on $V(G)$ is in $\eq_c(G)$ if for any equivalence class $C\in V(G)/\sim$ the graph $G_{\mid C}$ induced by $C$ is a connected graph. So, for any equivalence on $V(G)$ under the  assumption of taking classes being connected, contraction and restriction give two graphs associated to $G$. 
\begin{remark}
For an edge $e=\{x,y\}\in E(G)$, let $\sim_e$ be the equivalence whose classes are $$V(G/\sim_e)=\{x,y\}\bigcup_{z\neq x,y} \{z\}.$$ Then $\sim_e \in \eq_c(G)$: for simplicity, we will denote by $G/e$ the contracted graph $G/\sim_e$.
\end{remark}

We thus define a second coproduct $\delta$ on $\calH_\bfG$ by
\begin{align*}
&\forall G\in\bfG,&\delta(G)&=\sum_{\sim\ \in \eq_c(G)} G/\sim\otimes \ G\mid \sim.
\end{align*}
This coproduct is coassociative, but not cocommutative. Its counit $\epsilon_\delta$ is given by
\begin{align*}
&\forall G\in\bfG,&\epsilon_\delta(G)&=\begin{cases}
1\mbox{ if }E(G)=\emptyset,\\
0\mbox{ otherwise}.
\end{cases}
\end{align*}

\begin{example}
\begin{align*}
\delta(\grun)&=\grun \otimes\grun,\\
\delta(\grdeux)&=\grdeux\otimes \grun\grun+ \grun\otimes \grdeux,\\
\delta(\grtroisun)&=\grtroisun\otimes \grun\grun\grun+ \grun\otimes \grtroisun
+3\grdeux\otimes \grdeux\grun,\\
\delta(\grtroisdeux)&=\grtroisdeux\otimes\grun\grun\grun+2\grdeux\otimes \grdeux\grun,\\
\delta(\grquatreun)&=\grquatreun\otimes\grun\grun\grun\grun+ \grun\otimes\grquatreun
+6\grtroisun \otimes \grdeux\grun\grun+\grdeux\otimes (6\grdeux\grdeux+4\grtroisun\grun),\\
\delta(\grquatredeux)&=\grquatredeux\otimes\grun\grun\grun\grun+ \grun\otimes\grquatredeux
+(4\grtroisun+\grtroisdeux)\otimes \grdeux\grun\grun
+\grdeux\otimes (2\grdeux\grdeux+2\grun\grtroisun+2\grun\grtroisdeux),\\
\delta(\grquatretrois)&=\grquatretrois\otimes\grun\grun\grun\grun+ \grun\otimes\grquatretrois
+(\grtroisun+3\grtroisdeux)\otimes \grdeux\grun\grun
+\grdeux\otimes (\grdeux\grdeux+\grun\grtroisun+2\grun\grtroisdeux),\\
\delta(\grquatrequatre)&=\grquatrequatre\otimes\grun\grun\grun\grun+ \grun\otimes\grquatrequatre
+4\grtroisun\otimes \grdeux\grun\grun +\grdeux\otimes (2\grdeux\grdeux+4\grun\grtroisun),\\
\delta(\grquatrecinq)&=\grquatrecinq\otimes\grun\grun\grun\grun+ \grun\otimes\grquatrecinq
+3\grtroisdeux\otimes \grdeux\grun\grun+3\grdeux\otimes \grun\grtroisdeux,\\
\delta(\grquatresix)&=\grquatresix\otimes\grun\grun\grun\grun+ \grun\otimes\grquatresix
+3\grtroisdeux\otimes \grdeux\grun\grun +\grdeux\otimes (\grdeux\grdeux+2\grun\grtroisdeux).
\end{align*}
\end{example}

\begin{prop}  \cite[Theorem 1.7]{Foissy36}
$(\calH_\bfG,m,\Delta,\delta)$ is a double bialgebra. 
\end{prop}

As $(\calH_\bfG,m,\Delta)$ is a graded and connected bialgebra, it is a Hopf algebra. Its antipode is denoted by $S$.
The invertible characters for the convolution $\star$ induced by $\delta$ are given by the following:

\begin{lemma}\cite[Theorem 2.1]{Foissy36}\label{lemmeinversibles}
Let $\lambda \in \chara(\calH_\bfG)$. It is invertible for the product $\star$ if, and only if, $\lambda(\tun)\neq 0$. 
\end{lemma}

For a set $X$  (the set of ``colors''), recall that a {\it proper $X$--coloring}  of a graph $G$ is an assignment of an element of  $X$  to each vertex  $f:V(G) \rightarrow X$ such that adjacent vertices are assigned different colors. The chromatic number $\chi(G)$ of $G$ is the minimum number of colors in a proper vertex coloring of $G$: 
\[\chi(G)= \min \{|f(V(G))|: f \mbox{ is a proper coloring of } G\}.\]
 For any $k\in \N$, we call {\it $k$--proper coloring} of $G$ any proper coloring using at most $k$ colors and denote by  
 $\phi_{chr}(G,k)$ the number of proper $k$-colorations of $G$. If $|V(G)| = 0$,  one takes this number to be  $1$. It is a well--known result that  $\phi_{chr}(G, k)$ is a polynomial in $k$ with integer coefficients:  it can be extended to a unique polynomial  $\phi_{chr}(G)=\phi_{chr}(G,x)\in K[x]$, called is the {\it chromatic polynomial},  an important invariant of graph theory. The problem of  coloring a graph can be reduced to the problem of coloring two graphs derived from $G$: for $e=\{a,b\}\in E(G)$, if  $G -  e $ is the graph obtained by $G$ removing the edge $e$, and $G/e$ is the graph obtained by $G$ by contraction of $e$, then the following holds: 
$$\phi_{chr}(G)=\phi_{chr}(G -  e)-\phi_{chr}(G/e).$$

\begin{prop}

\begin{enumerate}
\item \cite[Theorem 2.4]{Foissy36} The unique double bialgebra morphism from the double algebra of graphs 
$(\calH_\bfG,m,\Delta,\delta)$ to $(\K[X],m,\Delta,\delta)$ (by virtue of Theorem \ref{theoaajouter}) is the chromatic polynomial
$\phi_{chr}$.
\item \cite[Proposition 2.2]{Foissy36} The following morphism is a Hopf algebra morphism from 
$(\calH,\bfG,m,\Delta)$ to $(\K[X],m,\Delta)$:
\[\phi_0=\left\{\begin{array}{rcl}
\bfH_\bfG&\longrightarrow&\K[X]\\
G\in \bfG&\longrightarrow&X^{|V(G)|}.
\end{array}\right.\]
\end{enumerate}\end{prop}

\subsection{Double bialgebra of oriented graphs}

An oriented graph is a pair $G=(V(G),A(G))$, where $V(G)$ is a finite set, called the set of vertices of $G$,
and $A(G)$ a set of ordered pairs of distinct elements of $V(G)$, i.e. 
$$A(G)\subseteq V(G)\times V(G)\setminus \{(x,x) | x\in V(G)\}.$$ The elements of $A(G)$ are called the arcs of $G$.
The set of (isoclasses of) oriented graphs is denoted by $\bfG_o$. (Note that we are not considering loops on oriented graphs, that is to say arcs which both extremities are equal).

Let us recall the double bialgebra structure on oriented graphs of \cite{Manchon2012}.
As a vector space, $\calH_{\bfG_o}$ is the vector space generated by $\bfG_o$. 
If $G,G'\in \bfG_o$, their product is the graph $GG'$ defined by
\begin{align*}
V(GG')&=V(G)\sqcup V(G'), &A(GG')&=A(G)\sqcup A(G').
\end{align*}
This induces a commutative and associative product $m$ on $\calH_{\bfG_o}$, which units is the empty graph $1$.\\

Let $G$ be an oriented graph and $I\subseteq V(G)$. The oriented subgraph $G_{\mid I}$ is defined by
\begin{align*}
V(G_{\mid I})&=I,&A(G_{\mid I})&=\{(x,y)\in A(G)\mid x,y\in I\}.
\end{align*}
We shall say that $I$ is an ideal of $G$ if for any $x,y\in V(G)$,
\[x\in I \mbox{ and }(x,y)\in A(G)\Longrightarrow y\in I.\]
These notions induce a commutative and coassociative coproduct $\Delta$ on $\calH_{\bfG_o}$ given by
\begin{align*}
&\forall G\in\bfG_o,&\Delta(G)&=\sum_{\mbox{\scriptsize $I$ ideal of $G$}}=G_{\mid I}\otimes G_{\mid V(G)\setminus I}.
\end{align*}
Its counit $\varepsilon_\Delta$ is given by
\begin{align*}
&\forall G\in\bfG_o,&\varepsilon_\Delta(G)&=\delta_{G,1}.
\end{align*}

Let $G$ be an oriented graph and let $\sim$ be an equivalence relation on $V(G)$. We define the contracted oriented graph $G/\sim$ by
\begin{align*}
V(G/\sim)&=V(G)/\sim,&A(G/\sim)&=\{(\pi_\sim(x),\pi_\sim(y))\mid (x,y)\in A(G), \:\pi_\sim(x)\neq \pi_\sim(y)\},
\end{align*}
where $\pi_\sim:V(G)\longrightarrow V(G)/\sim$ is the canonical surjection. 
We define the restricted oriented graph $G\mid\sim$ by
\begin{align*}
V(G\mid \sim)&=V(G),&A(G\mid \sim)=\{(x,y)\in A(G)\mid \pi_\sim(x)=\pi_\sim(y)\}.
\end{align*}
In other words, $G\mid \sim$ is the disjoint union of the oriented subgraphs $G_{\mid \pi}$, with $\pi\in V(G)/\sim$.
We shall say that $\sim\eq_c(G)$ if for any class $C\in V(G)/\sim$, $G_{\mid C}$ is a connected graph.
We then define a second coproduct $\delta$ on $\calH_{\bfG_o}$ by
\begin{align*}
&\forall G\in\bfG_o,&\delta(G)&=\sum_{\sim\in \eq_c(G)} G/\sim\otimes G\mid \sim.
\end{align*}
This coproduct is coassociative, but not cocommutative. Its counit $\epsilon_\delta$ is given by
\begin{align*}
&\forall G\in\bfG_o,&\epsilon_\delta(G)&=\begin{cases}
1\mbox{ if }E(G)=\emptyset,\\
0\mbox{ otherwise}.
\end{cases}
\end{align*}

\begin{theo}\cite{Manchon2012}
$(\calH_{\bfG_o},m,\Delta,\delta)$ is a double bialgebra. 
\end{theo}

\subsection{The antipode for graphs}

\begin{defi}
A mixed graph is a triple $G=(V(G),E(G),A(G))$, such that:
\begin{itemize}
\item $V(G)$ is a finite set, whose elements are the vertices of $G$.
\item $E(G)$ is a set of unordered pairs of elements of $V(G)$, whose elements are called the edges of $G$.
\item $A(G)$ is a set of ordered pairs of elements of $V(G)$, made of distinct elements, called the arcs of $G$.
\end{itemize}
We assume that the following conditions hold: for any $x,y\in V(G)$, distinct,
\begin{itemize}
\item If $\{x,y\}\in E(G)$, then $(x,y)\notin A(G)$ and $(y,x)\notin A(G)$.
\item If $(x,y)\in A(G)$, then $(y,x)\notin A(G)$ and $\{x,y\}\notin E(G)$.
\end{itemize}
\end{defi}

Mixed graphs have interesting applications.  For example, in operational research  they may be used to model the so--called ``job--shop scheduling problems'', in which a collection of tasks is to be performed, subject to certain timing constraints: undirected edges represent  two tasks that cannot be performed simultaneously, and directed edges represent precedence constraints, when one task must be performed before another one. A different example comes from Bayesian inference, where acyclic mixed graphs (that is a graph with no cycles of directed arcs) are used: undirected edges indicate a non--causal correlation between two events;  directed edges indicate a causal correlation in which the outcome of the first event influences the probability of the second event.

\begin{defi}
Let $H$ be a mixed graph. We associate to it two simple graphs $\gr(H)$  and $\gr_0(H)$, 
 defined  respectively removing  the orientations of the oriented  edges, and removing the oriented edges, that is: 
  \begin{align*}
V(\gr(H))&=V(H),&
E(\gr(H))&=E(H)\cup\{\{x,y\}\mid (x,y)\in A(H)\},\\
V(\gr_0(H))&=V(H),&
E(\gr_0(H))&=E(H).
\end{align*}
\end{defi}

\begin{defi}
Let $G\in {\bf G}$. A partial orientation of $G$ is a mixed graph $H$ such that $\gr(H)=G$. 
We shall say that a partial orientation is not totally acyclic if there exists a sequence $(x_0,\ldots,x_n)$ 
of vertices of $G$ such that:
\begin{itemize}
\item $n\geq 2$.
\item $x_0=x_n$.
\item For any $i\in [n]$, $\{x_{i-1},x_i\}\in E(H)$ or $(x_{i-1},x_i)\in A(H)$.
\item There exists at least one $i\in [n]$ such that  $(x_{i-1},x_i)\in A(H)$.
\end{itemize}
The set of totally acyclic partial orientations of $G$ is denoted by $\potac(G)$. 
\begin{remark}
For any graph $G$, $G\in \potac(G)$.
\end{remark}
\end{defi}

\begin{defi}
An orientation of a simple graph $G$  is an assignment of one and only one  ordering (or direction) to each edge 
$\{u,v\}$, denoted by $(u,v)$   or $(v,u) $, as the case may be. In other words, an orientation of $G$ is a mixed graph $H$ where $E(H)=\emptyset$ and $gr(H)=G$. An orientation $H$ of $G$ is said to be acyclic if it has no directed cycles, that is if there is no sequence $(x_0,\ldots,x_n)$ of vertices of $G$ such that:
\begin{itemize}
\item $n\geq 2$.
\item $x_0=x_n$.
\item For any $i\in [n]$, $(x_{i-1},x_i)\in A(H)$.
\end{itemize}
\end{defi}

\begin{prop} \label{propantipodegraphs}
The antipode $S$ of $(\calH_\bfG,m,\Delta)$ is given by
\begin{align*}
&\forall G\in \bfG,&S(G)&=\sum_{H\in \potac(G)} (-1)^{\cc(\gr_0(H))} \gr_0(H),
\end{align*}
 For any graph $K\in \bfG$, $\cc(K)$ is the number of connected components of $K$.
\end{prop}

\begin{proof}
We make use of a well known result by Stanley \cite{Stanley1973}. It states that  the chromatic polynomial of a graph evaluated in $-1$ is up to sign the number of its acyclic orientations, more precisely: 
$$\phi_{chr}(G)(-1)=(-1)^{|V(G)|}|\{\mbox{acyclic orientations of $G$}\}|.$$
Let us apply Theorem \ref{theoantipode}. The morphism $\phi_{chr}$ is a double bialgebra morphism from
$(\calH_\bfG,m,\Delta,\delta)$ to $(\K[X],m,\Delta,\delta)$, so, for any $G\in \bfG$,
\[\epsilon_\delta^{*-1}(G)=\phi_{chr}(G)(-1)
.\]
Hence, 
\[S(G)=\sum_{\sim\in \eq_c[G]} (-1)^{\cl(\sim)}  |\{\mbox{acyclic orientations of $G/\sim$}\}| G\mid \sim.\]
Let us consider the set 
\[\calA=\bigsqcup_{\sim\in \eq_c[G]}\{\mbox{acyclic orientations of $G/\sim$}\}.\]
Let $H=(V(G),E(H),A(H))\in \potac(G)$. We denote by $\sim_H$ the equivalence whose classes are the
connected components of the graph $\gr_0(H)$. If $\pi\in V(G)/\sim_H$, then $H_{\mid \pi}$ is connected.
As $G_{\mid \pi}$ has more edges than $H_{\mid \pi}$ (because the edges of $H$ are edges of $G$),
it is connected, so $\sim_H\in \eq_c[G]$. Let us assume that $(x,y)\in A(H)$, and let $x'\sim_H x$, $y'\sim_H y$, such
that $\{x,y\}\in E(G)$. There exist non oriented paths $(y,y_1,\ldots,y_k,y')$  and $(x',x'_1,\ldots,x'_l,x)$ in $H$. 
If $\{x',y'\}\in E(H)$ or $(y',x')\in A(H)$, then the sequence $(x,y,y_1,\ldots,y_k,y',x',x'_1,\ldots,x'_l,x)$
proves that $H$ is not totally acyclic: this is a contradiction. So $(x',y')\in A(H)$. Hence, $H$ induces a total orientation
of $G/\sim_H$, which we denote by $H/\sim_H$. It is acyclic: if $(\overline{x_0},\ldots,\overline{x_n},\overline{x_0})$ is an
oriented cycle in $H/\sim_H$, there exist sequences $(x'_0,\ldots,x'_n,y'_0)$ and $(x''_0,\ldots,x''_n,y''_0)$ such that:
\begin{itemize}
\item For any $i$, $x'_i\sim_H x''_i \sim_H x_i$ and $y'_0\sim_H y''_0 \sim_H x_0$.
\item For any $i$, $(x''_i,x'_{i+1}) \in A(H)$ and $(y'_0,y''_0)\in A(H)$.
\end{itemize}
By definition of $\sim_H$, there exists a non oriented path in $H$ from $x'_i$ to $x''_i$ for any $i$ and from $y''_0$ to $x'_0$. 
Hence, we obtain a cycle in $H$, containing at least one arc of $H$: $H$ is not totally acyclic, this is a contradiction.
So $H/\sim_H$ is acyclic. We obtain a map
\[\theta:\left\{\begin{array}{rcl}
\potac(G)&\longrightarrow&\calA\\
H&\longmapsto&H/\sim_H.
\end{array}\right.\]

Let us prove that $\theta$ is injective. Let $H_1,H_2\in \potac(G)$ such that $H_1/\sim_{H_1}=H_2/\sim_{H_2}$.
Then $\sim_{H_1}=\sim_{H_2}$. Let $(x,y)\in A(H_1)$. Then $(\overline{x},\overline{y})\in A(H_1/\sim_{H_1})=
A(H_2/\sim_{H_2})$, so $(x,y)\in A(H_2)$. By symmetry, $A(H_1)=A(H_2)$, so $H_1=H_2$. 

Let us prove that $\theta$ is surjective. Let $\sim\in \eq_c[G]$ and $\overline{H}$ be an acyclic orientation of $G/\sim$.
We define a partial orientation $H$ of $G$ as follows:
\[A(H)=\{(x,y)\mid \{x,y\}\in E(G),\: (\overline{x},\overline{y})\in A(\overline{H})\}.\]
It is totally acyclic: any cycle in $H$ containing at least one arc induces an oriented cycle in $\overline{H}$
of length at least two, which is not possible as $\overline{H}$ is acyclic. 
As $\sim\in \eq_c[G]$, its classes are the connected components of $G\mid \sim$, which is equal to $\gr_0(H)$ by construction
of $H$. Hence, $\sim_H=\sim$ and $\theta(H)=\overline{H}$.\\
Finally,
\begin{align*}
S(G)&=\sum_{\sim\in \eq_c[G]} (-1)^{\cl(\sim)}  |\{\mbox{acyclic orientations of $G/\sim$}\}| G\mid \sim\\
&=\sum_{\sim\in \eq_c[G]} (-1)^{\cc(G\mid \sim)}  |\{\mbox{acyclic orientations of $G/\sim$}\}| G\mid \sim\\
&=\sum_{H\in \potac(G)}(-1)^{\cc(\gr_0(H))} \gr_0(H). \qedhere
\end{align*}
\end{proof}

\subsection{The Fortuin and Kasteleyn's polynomial as a Hopf morphism}

The  Fortuin--Kasteleyn polynomial, which is due to Fortuin \cite{Fortuin1971}, and  Fortuin and Kasteleyn \cite{Fortuin1972}, is a two variable polynomial, comes from a random cluster model, and it is a variant of the Tutte polynomial. We recall here both the definitions.

\begin{notation} For $G\in \bfG$, if $F\subseteq E(G)$, denote by $G_{\mid F}$ the subgraph of $G$ defined by
\begin{align*}
V(G_{\mid F})&=V(G),&E(G_{\mid F})&=F.
\end{align*}
\end{notation}

\begin{defi}
Let $G\in \bfG$. 
\begin{enumerate}
\item A spanning graph of $G$ is a graph $H$ with $V(H)=V(G)$ and $E(H)\subseteq E(G)$.
 The set of spanning graphs of $G$ is denoted by $\spann(G)$.
 The set of spanning forests of $G$ (that is, spanning graphs of $G$ which are acyclic) is denoted by $\spanfor(G)$.  
\item A covering graph of $G$ is a spanning graph $H$ of $G$ such that $\cc(H)=\cc(G)$. 
The set of covering graphs of $G$ is denoted by $\cov(G)$. 
The set of covering forests of $G$ (that is, covering graphs of $G$ which are forests)  is denoted by $\covfor(G)$. 
\end{enumerate}
\end{defi}

\begin{remark}\begin{enumerate}
\item Let $G\in \bfG$ and $H\in \spann(G)$.
The connected components of $G$ are disjoint  union of connected components of $H$,
so $\cc(H)\geq \cc(G)$. If $H$ is a covering graph of $G$, then the connected components of $G$ and $H$
are the same.
\item The spanning graphs of $G$ are the graphs $G_{\mid F}$, with $F\subseteq E(G)$. 
Therefore, for any graph $G$, $|\spann(G)|=2^{|E(G)|}$. 
\end{enumerate}
\end{remark}

\begin{defi}
To any graph $G\in \bfG$, we associate  the Fortuin-Kasteleyn polynomial $Z_G(X,Y)\in \K[X,Y]$ defined by
\[Z_G(X,Y)=\sum_{F\subseteq E(G)} X^{\cc(G_{\mid F})} Y^{|F|}.\]
\end{defi}
Recall that the rank of a graph $G\in \bfG$ is the number $r(G)=|V(G)|-\cc(G)$; the nullity $n(G)$ of $G$ is defined by the relation $n(G)+r(G)=|E(G)|$. 

\begin{defi}
The rank--generating polynomial associated to $G\in \bfG$  is:
$$S_G(X,Y)=\sum_{F\subseteq E(G)} X^{r(G)-r(G_{\mid F})} Y^{n(G_{\mid F})}.$$
The Tutte polynomial $T_G(X,Y)\in \K[X,Y]$ is a simple function of the rank--generating polynomial:
\begin{align*}
T_G(X,Y)&=S_G(X-1,Y-1)\\
&=\sum_{F\subseteq E(G)} (X-1)^{-\cc(G)+\cc(G_{\mid F})} (Y-1)^{\cc(G_{\mid F})+|F|-|V(G)|}\\
&=(X-1)^{-\cc(G)}(Y-1)^{-|V(G)|}\sum_{F\subseteq E(G)}
\left((X-1)(Y-1)\right)^{\cc(G_{\mid F})} (Y-1)^{|F|}.
\end{align*}
\end{defi}

\begin{remark}
\begin{enumerate}
\item 
From the previous equality on the Tutte polynomial and from the definition of the Fortuin--Kasteleyn polynomial, it follows that the two polynomials satisfy the following relations:
\begin{align}
\label{EQ1} T_G(X,Y)=(X-1)^{-\cc(G)} (Y-1)^{-|V(G)|} Z_G((X-1)(Y-1),Y-1);
\end{align}
equivalently, 
\begin{align}
\label{EQ2} Z_G(X,Y)&=X^{\cc(G)}Y^{|V(G)|-\cc(G)}T_G\left(\frac{X}{Y}+1,Y+1\right).
\end{align}
\item It should be noted that $Z_G$ and $T_G$ are classically defined on graphs with multiple edges and with loops. However for the purpose of the Hopf structure we are interested in,  they do not have significance in the coproduct (and we restrict our attention to simple graphs). Tutte's polynomial satisfies a recursion which uses the ''deletion/contraction`` of edges, similarly to the chromatic polynomial. More precisely, starting from the value $T_{E_n}(X,Y)=1$ (with $E_n$ the empty graph),  one has
\begin{align*}
T_G(X,Y)&=\left\{\begin{array}{ll}
XT_{G - e}&\mbox{  if } e \mbox{  is a bridge},\\
YT_{G - e}&\mbox{  if } e \mbox{  is a loop},\\
T_{G - e} + T_{G/e} &\mbox{  if } e \mbox{  is neither a bridge nor a loop},
\end{array}\right.
\end{align*}
where a  bridge for $G$ is an edge $e\in E(G)$ such that its deletion increases the number of connected components.
\end{enumerate}
\end{remark}

\begin{theo}
Let $y\in \K$. We put
\begin{align*}
\zeta_y&:\left\{\begin{array}{rcl}
\calH_\bfG&\longrightarrow&\K[X]\\
G\in \bfG&\longrightarrow&Z_G(X,y).
\end{array}\right.\\
\lambda_y&:\left\{\begin{array}{rcl}
\calH_\bfG&\longrightarrow&\K\\
G\in \bfG&\longrightarrow&\displaystyle \sum_{H\in \cov(G)} y^{|E(H)|}.
\end{array}\right.\\
\mu_y&:\left\{\begin{array}{rcl}
\calH_\bfG&\longrightarrow&\K\\
G\in \bfG&\longrightarrow&(1+y)^{|E(G)|}.
\end{array}\right.
\end{align*}
Then $\zeta_y$ is a Hopf algebra morphism from $(\calH_\bfG,m,\Delta)$ to $(\K[X],m,\Delta)$
and  $\lambda_y$ and $\mu_y$ are characters of $\calH_\bfG$. Moreover,
\begin{align*}
\zeta_y&=\phi_0\leftsquigarrow \lambda_y=\phi_{chr}\leftsquigarrow \mu_y.
\end{align*}
\end{theo}

\begin{proof}
Let $G_1,G_2\in \bfG$. Then, for the set of covering graphs of the product, one has
\[\cov(G_1G_2)=\{H_1H_2\mid H_1\in \cov(G_1),\: H_2\in \cov(G_2)\}.\]
This implies that $\lambda_y$ is a character. Obviously, $\mu_y$ is a character.\\

Let $G\in \bfG$  and let $F\subseteq E(G)$. We define an equivalence $\sim_F\in \eq[V(G)]$
such that the classes of $\sim_F$ are the connected components of $G_{\mid F}$. 
Let $\pi$ be a class of $\sim_F$. Then $(G_{\mid F})_{\mid \pi}$ is connected by definition.
As this graph has less edges than $G_{\mid \pi}$, the latter is connected. So $\sim_F\in \eq_c[G]$.
Moreover, $G_{\mid F}$ is a covering graph of $G\mid \sim_F$, as the connected components of $G\mid \sim_F$
are the classes of $\sim_F$, that is to say the connected components of $G_{\mid F}$. 
Therefore,
\begin{align*}
\zeta_y(G)&=\sum_{F\subseteq E(G)} X^{\cc(G_{\mid F})} y^{|F|}\\
&=\sum_{\sim\in \eq_c[G]}\sum_{H\in  \cov(G\mid \sim)} X^{\cc(H)} y^{|E(H)|}\\
&=\sum_{\sim\in \eq_c[G]}\sum_{H\in  \cov(G\mid \sim)} X^{\cc(G\mid \sim)} y^{|E(H)|}\\
&=\sum_{\sim\in \eq_c[G]}\sum_{H\in  \cov(G\mid \sim)} X^{|V(G/\sim)|} y^{|E(H)|}\\
&=\sum_{\sim\in \eq_c[G]} \phi_0(G/\sim) \lambda_y(G\mid \sim)\\
&=(\phi_0\leftsquigarrow \lambda_y)(G).
\end{align*}
So $\zeta_y=\phi_0\leftsquigarrow \lambda_y$. As $\phi_0$ is a Hopf algebra morphism and $\lambda_y$ is a character,
$\zeta_y$ is a Hopf algebra morphism. Therefore, by Theorem \ref{theoaajouter}, there exists a unique character 
$\mu_y\in \Char(\calH_\bfG)$ such that $\zeta_y=\phi_{chr}\leftsquigarrow \mu_y$. Still by Theorem \ref{theoaajouter}, 
this character is $\mu_y=\epsilon_\delta\circ \zeta_y$: for any $G\in \bfG$,
\begin{align*}
\mu_y(G)&=\zeta_y(G)(1)=\sum_{F\subseteq E(G)}y^{|F|}=(1+y)^{|E(G)|}.\qedhere
\end{align*} \end{proof}

\begin{remark}
As $\zeta_y$ is a Hopf algebra morphism for any $y\in \K$, identifying $\K[X]\otimes \K[X]$ and $\K[X_1,X_2]$, we obtain that for any graph $G\in \bfG$,
\[Z_G(X_1+X_2,Y)=\sum_{V(G)=I\sqcup J} Z_{G_{\mid I}}(X_1,Y)Z_{G_{\mid J}}(X_2,Y).\]
\end{remark}

\begin{prop}\label{propZchromatique}
For any graph $G\in \bfG$,
\begin{align*}
Z_G(X,-1)&=\phi_{chr}(G),\\
Z_G(X,0)&=\phi_0(G),\\
\phi_{chr}(G)&=(-1)^{|V(G)|+\cc(G)} X^{\cc(G)}T_G(1-X,0).
\end{align*}
\end{prop}

\begin{proof}
For any graph $G\in\bfG$,
\[\mu_{-1}(G)=\begin{cases}
1\mbox{ if }E(G)=\emptyset,\\
0\mbox{ otherwise},
\end{cases}\]
so $\mu_{-1}=\epsilon_\delta$. Therefore, 
\[\zeta_{-1}=\phi_{chr}\leftsquigarrow \epsilon_\delta=\phi_{chr}.\]
Using (\ref{EQ2}), we obtain the relation between $\phi_{chr}$ and $T_G(1-X,0)$. 
Moreover, if the graph $G\in \bfG$ has $E(G)\neq \emptyset$, then for any $H\in \cov(G)$ one has $E(H)\neq \emptyset$. Therefore, 
\[\lambda_0(G)=\begin{cases}
1\mbox{ if }E(G)=\emptyset,\\
0\mbox{ otherwise},
\end{cases}\]
so $\lambda_0=\epsilon_\delta$. Consequently,
\[\zeta_0=\phi_0\leftsquigarrow \epsilon_\delta=\phi_0. \qedhere\]
\end{proof}

\begin{prop}
The character $\mu_0$ is invertible for the convolution $\star$ and $\mu_0^{\star-1}=\lambda_{-1}$. 
For any $y\in \K$, 
\begin{align*}
\mu_y&=\mu_0\star \lambda_y,&\lambda_y&=\lambda_{-1}\star \mu_y.
\end{align*}
\end{prop}

\begin{proof}
For any $G\in \bfG$,
\[\epsilon_\delta \circ \phi_0(G)=1^{|V(G)|}=1=\mu_0(G),\]
so $\epsilon_\delta\circ \phi_0=\mu_0$. Moreover, 
\begin{align*}
\mu_y&=\epsilon_\delta \circ \zeta_y\\
&=\epsilon_\delta \circ (\phi_0\otimes \lambda_y)\circ \delta\\
&=((\epsilon_\delta \circ \phi_0)\otimes \lambda_y)\circ \delta\\
&=(\mu_0\otimes \lambda_y)\circ \delta\\
&=\mu_0\star \lambda_y.  
\end{align*}
As noticed in the proof of Proposition \ref{propZchromatique}, $\mu_{-1}=\epsilon_\delta$,  so 
$\epsilon_\delta=\mu_0\star \lambda_{-1}$. 
As $\mu_0(\grun)=1$, $\mu_0$ is invertible for $\star$ (Lemma \ref{lemmeinversibles}) and $\mu_0^{\star-1}=\lambda_{-1}$. 
\end{proof}

\begin{prop}\label{propantipodeZG}
For any graph $G$,
\begin{align*}
Z_G(-X,Y)&=\sum_{H\in \potac(G)} (-1)^{\cc(\gr_0(H))} Z_{\gr_0(H)}(X,Y),\\
T_G(2-X,Y)&=\sum_{H\in \potac(G)} (1-X)^{\cc(\gr_0(H))-\cc(G)} T_{\gr_0(H)}(X,Y).
\end{align*}
\end{prop}

\begin{proof}
For any $y\in \K$, $\zeta_y:(\calH_\bfG,m,\Delta)\longrightarrow (\K[X],m,\Delta)$ is a bialgebra morphism;
as both $(\calH_\bfG,m,\Delta)$ and $(\K[X],m,\Delta)$ are Hopf algebras, $\zeta_y$ is a Hopf algebra morphism,
that is to say $S\circ \zeta_y=\zeta_y\circ S$. For any graph $G$, this gives, with Proposition \ref{propantipodegraphs},
\begin{align*}
Z_G(-X,y)&=S\circ \zeta_y(G)\\
&=\zeta_y\circ S(G)\\
&\zeta_y\left(\sum_{H\in \potac(G)} (-1)^{\cc(\gr_0(H))} \gr_0(H)\right)\\
&=\sum_{H\in \potac(G)} (-1)^{\cc(\gr_0(H))} Z_{\gr_0(H)}(X,y),
\end{align*}
which implies the first result. Then, 
\begin{align*}
T_G(2-X,Y)&=(X-1)^{-\cc(G)}(Y-1)^{-|V(G)|} Z_G((X-1)(1-Y),1-Y)\\
&=(-1)^{-\cc(G)}(1-X)^{-\cc(G)} (Y-1)^{-|V(G)|} Z_G(-(1-X)(1-Y),1-Y)\\
&=\sum_{H\in \potac(G)} (-1)^{\cc(\gr_0(H))-\cc(G)}(1-X)^{-\cc(G)} (Y-1)^{-|V(G)|}\\
& Z_{\gr_0(H)}((1-X)(1-Y),1-Y)\\
&=\sum_{H\in \potac(G)} (-1)^{\cc(\gr_0(H))-\cc(G)}(1-X)^{\cc(\gr_0(H))-\cc(G)} T_{\gr_0(H)}(X,Y)\\
&=\sum_{H\in \potac(G)} (X-1)^{\cc(\gr_0(H))-\cc(G)} T_{\gr_0(H)}(X,Y). \qedhere
\end{align*}
\end{proof}

\section{Combinatorial interpretations}\label{sec:CombInterp}

\subsection{For the Tutte polynomial}

\begin{lemma}
For any graph $G\in \bfG$, for any $y\in \K$,
\[\lambda_y(G)=y^{|V(G)|-\cc(G)} T_G(1,1+y).\]
\end{lemma}

\begin{proof}
By (\ref{EQ2}),
\[\lim_{X\longrightarrow 0} \frac{\zeta_y(G)}{X^{\cc(G)}}=y^{|V(G)|-\cc(G)} T_G(1,1+y).\]
Moreover,
\[\zeta_y(G)=(\phi_0\leftsquigarrow \lambda_y)(G)=\sum_{\sim\in \eq_c[G]} X^{\cl(\sim)} \lambda_y(G\mid \sim).\]
If $\sim\in \eq_c[G]$, then $\cl(\sim)\geq \cc(G)$, as the classes of $\sim$ as connected, 
and $\cl(\sim)=\cc(G)$ if, and only if, $\sim$ is the equivalence $\sim_c$ whose classes are the connected components of $G$.
Therefore,
\[\lim_{X\longrightarrow 0} \frac{\zeta_y(X)}{X^{\cc(G)}}=\lambda_y(G\mid \sim_c)=\lambda_y(G).\qedhere\]
\end{proof}

\begin{prop}
For any graph $G\in \bfG$,
\begin{align*}
T_G(1,2)&=|\cov(G)|,& T_G(1,1)&=|\covfor(G)|.
\end{align*}
\end{prop}

\begin{proof}
For $y=1$, by the previous Lemma one has $T_G(1,2)=\lambda_1(G)=\displaystyle \sum_{H\in \cov(G)} 1^{|E(H)|}
=|\cov(G)|$. 
Let $G$ be a covering graph of $G$. Then $G$ has at least $|V(G)|-\cc(G)$ edges, with an equality if, and only if,
$G$ is a covering forest of $G$. Hence,
\[\lambda_y(G)=y^{|V(G)|-\cc(G)}|\covfor(G)|+\mathcal{O}\left(y^{|V(G)|-\cc(G)+1}\right).\]
This implies
\begin{align*}
\lim_{y\longrightarrow 0} \frac{\lambda_y(G)}{y^{|V(G)|-\cc(G)}}&=|\covfor(G)|
=\lim_{y\longrightarrow 0}T_G(1,1+y)=T_G(1,1). \qedhere
\end{align*}
\end{proof}

\begin{prop}
For any graph $G\in \bfG$, $T_G(2,1)=|\spanfor(G)|$. 
\end{prop}

\begin{proof}
For any $x\in \K$, by (\ref{EQ2}),
\[\zeta_x(G)(x)=x^{\cc(G)}x^{|V(G)|-\cc(G)}T_G(2,1+x)=x^{|V(G)|} T_G(2,1+x).\]
Moreover,
\begin{align*}
\zeta_x(G)(x)&=(\phi_0\leftsquigarrow \lambda_x)(G)(x)
=\sum_{\sim\in \eq_c[G]} \sum_{H\in \cov(G\mid \sim)} x^{\cl(\sim)+|E(H)|}.
\end{align*}
Let $\sim \in \eq_c[G]$ and $H\in \cov(G\mid \sim)$. Then $H$ has $\cc(G\mid \sim)=\cl(\sim)$
connected components, so $|E(H)|\geq |V(G)|-\cc(\sim)$ and $\cl(\sim)+|E(H)|\geq |V(G)|$,
with equality if, and only if $H$ is a forest. Conversely, if $F\in \spanfor(G)$, 
denoting $\sim_F$ the equivalence whose classes are the connected components
of $G$, then $\sim_F\in \eq_c[G]$ and, moreover,  $F\in \cov(G\mid \sim_F)$, contributing with $x^{|V(G)|}$:
\[\zeta_x(G)(x)=|\spanfor(G)|x^{|V(G)|}+\mathcal{O}\left(x^{|V(G)|}\right).\]
Finally,
\[\lim_{x\longrightarrow 0}\dfrac{\zeta_x(G)(x)}{x^{|V(G)|}}=|\spanfor(G)|=T_G(2,1). \qedhere\]
\end{proof}

\subsection{For the Fortuin and Kasteleyn's polynomial}

\begin{prop}
Let $G$ be a graph and let $(x,y)\in \N\times \N$.
 A compatible $(x,y)$-pair of colorings of $G$ is a pair of maps $(c_V,c_E)$ 
\begin{align*}
c_V&:V(G)\longrightarrow \{1,\ldots,x\},&
c_E&:E(G)\longrightarrow\{0,\ldots,y\},
\end{align*}
such that for any $e=\{v,w\}\in E(G)$, $c_E(e)\neq 0$ if, and only if  $c_V(v)=c_V(w)$.
We denote by $\pc_{x,y}(G)$ the set of compatible $(x,y)$-pairs of colorings of $G$. 
For any $x\in \N$, for any $y\in \Z_{\geq -1}$, 
\[Z_G(x,y)=|\pc_{x,y+1}(G)|.\]
For any $x\in \N$,  for any $y\in \N_{\geq 1}$,
\[Z_G(x,-y)=\sum_{(c_V,c_E)\in \pc_{x,y-1}(G)} (-1)^{|c_E^{-1}([y-1])|}.\]
\end{prop}

\begin{proof}
As $\zeta_y=\phi_{chr}\leftsquigarrow \mu_y$, for any graph $G$,
\begin{align*}
Z_G(x,y)&=\sum_{\sim\in \eq_c[G]} \phi_{chr}(G/\sim)(x)(1+y)^{|E(G\mid \sim)|}\\
&=\sum_{\sim\in \eq_c[G]} |\{\mbox{proper $x$-colorings of $G/\sim$}\}|(1+y)^{|E(G\mid \sim)|}.
\end{align*}
Let $f=(c_V,c_E)$ be a compatible $(x,z)$-pair of colorings of $G$, with $z\in \N$.
We define an equivalence $\sim'_f$ on $V(G)$ by
\[v\sim'_f w \Longleftrightarrow c_V(v)=c_V(w).\]
This has no reason to be in $\eq_c[G]$: we now define $\sim_f$ as the equivalence whose classes are the connected components
of $G\mid \sim'_f$. Then $\sim_f\in \eq_c[G]$. As $c_V$ is constant on the classes of $\sim_f$, 
it induces a (vertex) coloring of $G/\sim_f$. This coloring is a proper vertex--coloring of $G$: if $\{C,D\}$ is an edge of $G/\sim_f$,
there exists an edge $\{v,w\}$ in $E(G)$, with $v\in C$ and $w\in D$. 
By absurd, if $c_V(v)=c_V(w)$, then $v,w$ are in the same connected component of $G\mid \sim_f$ as they are related by an edge,
so $v\sim_f w$: a contradiction, $C\neq D$. Hence, $c_V(v)\neq c_V(w)$. 

Conversely, if $\sim\in \eq_c[G]$ and $f$ is a proper $x$-coloring of $G/\sim$, $f$ can be extended to 
a map of all vertices $c_C:V(G)\longrightarrow [x]$, assigning to each vertex in a class of $\sim$ the color it has in the coloring of $G/\sim$. By the condition 
of compatibility  for a pair of (vertex, edge)-coloration, there exist exactly $z^{|E(G\mid\sim)|}$
maps $c_E$ completing $c_C$ to a compatible $(x,1+y)$ pair of colorings. 
Hence, $Z_G(x,y)$ is the number of compatible $(x,y+1)$-pairs of colorings of $G$. 

Similarly, if $x,y\in \N$,
\begin{align*}
Z_G(x,-y)&=\sum_{\sim\in \eq_c[G]} \phi_{chr}(G/\sim)(x)(1-y)^{|E(G\mid \sim)|}\\
&=\sum_{\sim\in \eq_c[G]} |(-1)^{E(G\mid \sim)}\{\mbox{proper $x$-coloring of $G/\sim$}\}|(y-1)^{|E(G\mid \sim)|}\\
&=\sum_{(c_V,c_E)\in \pc_{x,y-1}(G)} (-1)^{|c_E^{-1}([y-1])|}. \qedhere
\end{align*}
\end{proof}

From (\ref{EQ2}), it follows:

\begin{cor}
Let $G\in \bfG$. For any $x,y\geq 2$,
\[T_G(x,y)=\frac{1}{(x-1)^{\cc(G)}(y-1)^{|V(G)|}}|\pc_{(x-1)(y-1),y}(G)|.\]
For any $x,y\geq 0$,
\[T_G(-x,-y)=\frac{(-1)^{\cc(G)+|V(G)|}}{(x+1)^{\cc(G)}(y+1)^{|V(G)|}}
\sum_{(c_V,c_E)\in \pc_{(1+x)(1+y),y}(G)} (-1)^{|c_E^{-1}([y])|}.\]
\end{cor}

\section{Orientations}\label{sec:Orientations}

\subsection{Orientations of graphs as a  Hopf algebra morphism}

\begin{notation}
Let $G\in \bfG$. We denote by $\ori(G)$ the set of orientations of $G$ and by $\oriac(G)$ the set of acyclic orientations
of $G$.  By definitions, $\ori(G)$ contains $2^{|E(G)|}$ oriented graphs.
\end{notation}

\begin{prop}
The following map is a bialgebra morphism:
\[\Theta:\left\{\begin{array}{rcl}
(\calH_\bfG,m,\Delta)&\longrightarrow&(\calH_{\bfG_o},m,\Delta)\\
G\in \bfG&\longrightarrow&\displaystyle \sum_{H\in \ori(G)} H.
\end{array}\right.\]
\end{prop}

\begin{proof}
Let $G,G'\in \bfG$. Then
\[\ori(GG')=\{HH'\mid H\in \ori(G),\: H'\in \ori(H')\}.\]
Hence, $\Theta(GG')=\Theta(G)\Theta(G')$. \\

Let $G\in \bfG$. We consider
\begin{align*}
\calA&=\{(H,I)\mid H\in \ori(G),\mbox{ $I$ ideal of $H$}\},\\
\calB&=\{(J,H',H'')\mid J\subseteq V(G),\: H'\in \ori(G_{\mid  V(G)\setminus J}),\:H''\in \ori(G_{\mid J})\}.
\end{align*}
There exists an obvious map
\[\upsilon:\left\{\begin{array}{rcl}
\calA&\longrightarrow&\calB\\
(H,I)&\longrightarrow&(I,H_{\mid V(G)\setminus I},H_{\mid I}).
\end{array}\right.\]
Let us prove that $\upsilon$ is injective. Let us assume that $\upsilon(H_1,I_I)=\upsilon(H_2,I_2)$. 
Then $I_1=I_2=J$. Let $e\in A(H)$. 
\begin{itemize}
\item If both extremities of $e$ are elements of $J$,
 as $(H_1)_{\mid J}=(H_2)_{\mid J}$, $e$ is oriented in the same way in $H_1$ and in $H_2$.
\item If both extremities of $e$ are not elements of $J$, as $(H_1)_{\mid V(G)\setminus J}=(H_2)_{\mid V(G)\setminus J}$,
 $e$ is oriented in the same way in $H_1$ and in $H_2$.
\item Otherwise, let us denote by $x$ the extremity of $e$ which is in $V(G)\setminus J$
and $y$ the extremity of $e$ which is in $J$. As $J$ is an ideal of $H_1$ and of $H_2$, necessarily
$e$ is oriented in $H_1$ and in $H_2$ from $x$ to $y$.
\end{itemize}
Therefore, $H_1=H_2$.

Let us prove that $\upsilon$ is surjective. Let $(J,H',H'')\in \calB$. We define an orientation $H$ of $G$ as follows:
if $e\in E(G)$,
\begin{itemize}
\item If both extremities of $e$ are elements of $J$,
then choose for $e$ the same orientation as in $H''$. 
\item If both extremities of $e$ are not elements of $J$, then choose for $e$ the same orientation as in $H'$.
\item Otherwise, let us denote by $x$ the extremity of $e$ which is in $V(G)\setminus J$
and $y$ the extremity of $e$ which is in $J$. Then orient $e$  from $x$ to $y$:  $(x,y)$. 
\end{itemize}
We obtain $H\in \ori(G)$, such that $H_{\mid V(G)\setminus J}=H'$ and $H_{\mid J}=H''$.
Moreover, by construction there is no arc in $H$ from a vertex belonging to $J$ to a vertex not belonging to $J$,
so $J$ is an ideal of $H$. Therefore, $(H,J)\in \calA$ and $\upsilon(H,J)=(J,H',H'')$. 

Using this bijection,
\begin{align*}
\Delta\circ \Theta(G)&=\sum_{(H,J)\in \calA} H_{\mid V(G)\setminus J}\otimes H_{\mid J}\\
&=\sum_{(J,H',H'')\in \calB} H'\otimes H''\\
&=\sum_{J\subseteq V(G)} \Theta(G_{\mid V(G)\setminus J})\otimes \Theta(G_{\mid J})\\
&=(\Theta\otimes \Theta)\circ \Delta(G). 
\end{align*}
So $\Theta$ is a bialgebra morphism.
\end{proof}

\begin{remark}
The map $\Theta$ is not compatible with $\delta$. For example,
\begin{align*}
\Theta(\grdeux)&=2\grdeuxo,&\Theta(\grtrois)&=2\grtroisoun+6\grtroisodeux
\end{align*}
so
\begin{align*}
(\Theta\otimes \Theta)\circ \delta(\grtrois)
&=(\Theta\otimes \Theta)(\grtrois\otimes \grun\grun\grun+\grun \otimes \grtrois+3\grdeux\otimes \grdeux\grun)\\
&=\Theta(\grtrois)\otimes \grun\grun\grun+\grun \otimes \Theta(\grtrois)+12\grdeuxo\otimes \grdeuxo\grun,\\
\delta\circ \Theta(\grtrois)&=2\delta(\grtroisoun)+6\delta(\grtroisodeux)\\
&=\Theta(\grtrois)\otimes \grun\grun\grun+\grun \otimes \Theta(\grtrois)+12\grdeuxo\otimes \grdeuxo\grun
+12 \cycle\otimes \grdeuxo\grun.
\end{align*}
\end{remark}

We denote by $I_{oc}$ the space of $\calH_{\bfG_o}$ generated by oriented graphs containing an oriented cycle. 
If $G$ is such a graph:
\begin{itemize}
\item For any oriented graph $H$, $GH$ has an oriented cycle. In other terms, $I_{oc}$ is an ideal.
\item Let $I$ be an ideal of $G$. If $I$ contains a vertex of the oriented cycle of $G$, then it contains all the vertices of the cycle,
as it is an ideal. Therefore, $G_{\mid I}$ or $G_{\mid V(G)\setminus I}$ has an oriented cycle. In other words,
$I_{oc}$ is a coideal for $G$.
\item Let $\sim\in \eq_c[G]$. If all the vertices of $G$ are $\sim$-equivalent, then $G\mid \sim$ has an oriented cycle.
Otherwise, the contraction $G/\sim$ has an oriented cycle. Moreover, $\epsilon_\delta(G)=0$, as $A(G)\neq \emptyset$,
since it has an oriented cycle. In other terms, $I_{oc}$ is a coideal for $\delta$.
\end{itemize}
As a consequence, the quotient $\calH_{\bfG_o}/I_{oc}$ which we identify with the space $\calH_{\bfG_{aco}}$
of oriented acyclic graphs,  inherits a double bialgebra structure such that the following map is a double bialgebra morphism:
\[\pi:\left\{\begin{array}{rcl}
\calH_{\bfG_o}&\longrightarrow&\calH_{\bfG_{aco}}\\
G&\longrightarrow&\begin{cases}
G\mbox{ if $G$ is acyclic},\\
0\mbox{ otherwise}.
\end{cases}
\end{array}\right.\]

\begin{prop}\cite[Theorem 1.7]{Foissy36}
The following is a double bialgebra morphism:
\[\Theta_{ac}=\pi\circ \Theta:\left\{\begin{array}{rcl}
\calH_\bfG&\longrightarrow&\calH_{\bfG_{aco}}\\
G&\longrightarrow&\displaystyle \sum_{H\in \oriac(G)}H.
\end{array}\right.\]
\end{prop}

\begin{defi}
Let $G$ be a graph and $H$ be an orientation of $G$. We shall say that $H$ is strongly connected if for any $x,y\in V(G)$, 
there exists an oriented path from $x$ to $y$ in $H$. 
The set of orientations of $G$ such that any connected component of $G$ is strongly connected is denoted by $\osc(G)$. 
\end{defi}

\begin{lemma}
Let us consider the characters of $\calH_{\bfG_o}$ defined by
\begin{align*}
\mu_1&:\left\{\begin{array}{rcl}
\calH_{\bfG_o}&\longrightarrow&\K\\
G\in \bfG_o&\longrightarrow&1,
\end{array}\right.&
\mu_{sc}&:\left\{\begin{array}{rcl}
\calH_{\bfG_o}&\longrightarrow&\K\\
G\in \bfG_o&\longrightarrow&\begin{cases}
(-1)^{\cc(G)}\mbox{ if any connected component of $G$}\\
\hspace{1cm}\mbox{is strongly connected},\\
0\mbox{ otherwise}.
\end{cases}
\end{array}\right.
\end{align*}
Denoting $*$ the convolution associated to $\Delta$, then $\mu_{sc}=\mu_1^{*-1}$.
\end{lemma}

\begin{proof}
Let $G\in \bfG_o$, different from $1$. We denote by  $I_1,\ldots,I_k$ its strongly connected components
and by $\sim$ the equivalence on $V(G)$ whose classes are $I_1,\ldots,I_k$.
For any $i$, $G_{\mid I_i}$ is strongly connected, so is connected: $\sim\in \eq_c[G]$. 

Let us prove that $G/\sim$ is acyclic. If $(I_{j_1},\ldots,I_{j_p})$ is an oriented cycle in $G/\sim$,
then for any $k$, there exists an arc from a vertex of $I_{j_k}$ to $I_{j_{k+1}}$, with the convention $j_{p+1}=j_1$.
As $G_{\mid J_l}$ is strongly connected for any $l$, we deduce the existence of an oriented cycle in $G$, 
which goes through $I_1,\ldots,I_p$. All the vertices of this cycle are in the same strongly connected component of $G$,
so belong to the same $I_k$: we deduce that $I_{j_1}=\ldots=I_{j_p}$, $(I_{j_1},\ldots,I_{j_p})$ is an oriented cycle in $G/\sim$:
this is a contradiction. So $G/\sim$ is acyclic. 

As $G/\sim$ is acyclic, it has wells, that is to say vertices with no outgoing arc. We denote by
$J_1,\ldots,J_m$ its wells. These elements $J_1,\ldots,J_m$ of $V(G/\sim)$ are equivalent classes of $\sim$, that is to say some of the $I_i$'s: $\{J_1,\ldots,J_m\}\subseteq \{I_1,\ldots,I_k\}$. \\

Let $J\subseteq V(G)$. Let us show that it is an ideal of $G$ such that $\mu_{sc}(G_{\mid I})\neq 0$
if, and only if, $J$ is a disjoint union of $J_j$'s. 

$\Longrightarrow$. If so, the connected components of $G_{\mid J}$ are strongly connected.
Let us assume that its intersection with $I_i$ contains a vertex $x$. Let $x\in I_i$. 
As $G_{\mid I_i}$ is strongly connected, $G$ contains an oriented path from $x$ to $y$.
Since $J$ is an ideal, $y\in J$. Therefore, $G$ is a disjoint union of $I_i$'s. Let us assume that one of the $I_i$'s included in $G$
is not a well of $G/\sim$. There exists $j\neq i$, such that $(I_i,I_j)$ is an arc in $G/\sim$. 
Therefore, there exists an arc between a vertex $x\in I_i$ and a vertex $y\in I_j$ in $G$.
As $x\in J$ and $J$ is an ideal, $y\in J$.  The connected components of $G_{\mid J}$ being strongly connected, it follows that  there exists an oriented path from $y$ to $x$ in $G_{\mid J}$, so $x$ and $y$ are in the same strongly connected component
of $G$ and finally $i=j$: this is a contradiction. So $I_i$ is a well of $G/\sim$, so is one of the $J_j$'s.

$\Longleftarrow$. Let $J$ be a disjoint union of $J_j$'s. Let $x\in J$ and $y\in V(G)$, with an arc between $x$ and $y$, and $j$ such that $x\in J_j$. If $y\notin J_j$, there is an arc in $G/\sim$ from $J_j$ to another vertex:
this contradicts the fact that $J_j$ is a well. So $y\in J_j\subseteq J$, and $J$ is an ideal of $G$.
Its connected components are obviously the $I_i$'s it contains, so they are strongly connected:
$\mu_{sc}(G_{\mid J})\neq 0$. \\

Then,
\begin{align*}
\mu_1*\mu_{sc}(G)&=\sum_{\mbox{\scriptsize $I$ ideal of $G$}} \mu_{sc}(G_{\mid I})
=\sum_{I'\subseteq [m]} \mu_{sc}\left(G_{\displaystyle \mid \bigcup_{i\in I'} J_i}\right)
=\sum_{I'\subseteq [m]}(-1)^{|I'|}
=0,
\end{align*}
as $m\geq 1$. Therefore, $\mu_1*\mu_{sc}=\varepsilon_\Delta$, i.e. $\mu_{sc}=\mu_1^{*-1}$.  
\end{proof}

\begin{prop}
We consider the character $\alpha$ defined by
\begin{align*}
\alpha=\mu_1^{*-1}&:\left\{\begin{array}{rcl}
\calH_\bfG&\longrightarrow&\K\\
G\in \bfG&\longrightarrow&(-1)^{\cc(G)}|\osc(G)|.
\end{array}\right.
\end{align*}
Then $\alpha$ is the inverse for the convolution associated to $\Delta$ of the character $\mu_1$ of $\calH_\bfG$.
\end{prop}

\begin{proof}
For any graph $G\in\bfG$, 
\[\mu_1(G)=2^{|E(G)|}=|\ori(G)|=\mu_1\circ \Theta(G),\]
so $\mu_1=\mu_1\circ \Theta$. Moreover,
\[\alpha(G)=\sum_{H\in \ori(G)}\mu_{sc}(H)=\mu_{sc}\circ \Theta(G),\]
so $\alpha=\mu_{sc}\circ \Theta$. Consequently, since $\Theta$ is a Hopf algebra morphism:
\begin{align*}
\mu_1*\alpha&=(\mu_1\otimes \mu_{sc})\circ (\Theta\otimes \Theta)\circ \Delta
=(\mu_1\otimes \mu_{sc})\circ \Delta\circ \Theta
=\varepsilon_\Delta\circ \Theta=\varepsilon_\Delta,
\end{align*}
so $\alpha=\mu_1^{*-1}$. 
\end{proof}

\begin{prop}
Let $G\in \bfG$. Then $T_G(0,2)$ is the number of strongly connected orientations of $G$.
\end{prop}

\begin{proof}
Note that for any $y\in \K$, as $\zeta_y$ is a Hopf algebra morphism,
\begin{align*}
Z_G(-X,y)&=S\circ \zeta_y(G)=\zeta_y\circ S(G),
\end{align*}
so 
\[Z_G(-1,y)=\zeta_y\circ S(G)(1)=\epsilon_\delta\circ \zeta_y\circ S(G)=\mu_y\circ S(G)=\mu_y^{*-1}(G).\]
In the particular case $y=1$,
\[Z_G(-1,1)=\mu_1^{*-1}(G)=\alpha(G).\]
By (\ref{EQ2}),
\[Z_G(-1,1)=(-1)^{\cc(G)}T_G(0,2),\]
which directly implies the result. 
\end{proof}

\subsection{Values at negative integers}

\begin{prop}
Let $x,y\in \N$. For any graph $G$, we denote by $\opc_{x,y}(G)$ the set of triples $(H,c_V,v_E)$ such that:
\begin{itemize}
\item $H$ is a totally acyclic partial orientation of $G$.
\item $c_V:V(G)\longrightarrow [x]$ and $c_E:E(H)\longrightarrow \{0,\ldots,y\}$ are maps such that for any
$\{v,w\} \in E(H)$, 
\[c_E(\{v,w\})\neq 0 \Longleftrightarrow c_V(v)=c_V(w).\]
\end{itemize}
Let $G\in \bfG$. For any $x\in \N$, for any $y\in \Z_{\geq-1}$,
\begin{align*}
Z_G(-x,y)&=\sum_{(H,c_V,c_E)\in \opc_{x,y+1}(G)} (-1)^{\cc(\gr_0(H))}.
\end{align*}
For any $x\in \N$, for any $y\in \N_{\geq 1}$,
\begin{align*}
Z_G(-x,-y)&=\sum_{(H,c_V,c_E)\in \opc_{x,y-1}(G)}  (-1)^{\cc(\gr_0(H))+|c_E^{-1}([y-1])|}.
\end{align*}
\end{prop}

\begin{proof}
Note that 
\[\opc_{x,y}(G)=\{(H,c_V,c_E)\mid H\in P\ori(G),\: (c_V,c_E)\in \pc_{x,y}(\gr_0(H))\}.\]
Therefore, using Proposition \ref{propantipodeZG},
\begin{align*}
Z_G(-x,y)&=\sum_{H\in \potac(G)} (-1)^{\cc(\gr_0(H))}Z_G(x,y)\\
&=\sum_{\substack{H\in \potac(G),\\ (c_V,c_E)\in \pc_{x,y+1}(\gr_0(H))}}  (-1)^{\cc(\gr_0(H))}\\
&=\sum_{(H,c_V,c_E)\in \opc_{x,y+1}(G)} (-1)^{\cc(\gr_0(H))},
\end{align*}
whereas
\begin{align*}
Z_G(-x,-y)&=\sum_{H\in \potac(G)} (-1)^{\cc(\gr_0(H))}Z_G(x,-y)\\
&=\sum_{\substack{H\in \potac(G),\\ (c_V,c_E)\in \pc_{x,y-1}(\gr_0(H))}}  (-1)^{\cc(\gr_0(H))+|c_E^{-1}([y-1])|}\\
&=\sum_{(H,v_C,v_E)\in \opc_{x,y-1}(G)}  (-1)^{\cc(\gr_0(H))+|c_E^{-1}([y-1])|}. \qedhere
\end{align*}
\end{proof}

From (\ref{EQ2}):

\begin{cor}
Let $G\in \bfG$. For any $x\in \N$, $y\in \N_{\geq 2}$,
\[T_G(-x,y)=\frac{(-1)^{\cc(G)}}{(x+1)^{\cc(G)}(y-1)^{|V(G)|}}
\sum_{(H,c_V,c_E)\in \opc_{(x+1)(y-1),y}(G)} (-1)^{\cc(\gr_0(H))}.\]
For any $x\in \N_{\geq 2}$, $y\in \N$,
\[T_G(x,-y)=\frac{(-1)^{|V(G)|}}{(x-1)^{\cc(G)}(y+1)^{|V(G)|}}
\sum_{(H,c_V,c_E)\in \opc_{(x-1)(y+1),y}(G)}  (-1)^{\cc(\gr_0(H))+|c_E^{-1}([y])|}.\]
\end{cor}

\bibliographystyle{amsplain}
\bibliography{biblio}

\end{document}